\documentclass[a4note,twopage,reqno,12pt]{amsart}
\usepackage[top=30mm,right=30mm,bottom=30mm,left=30mm]{geometry}

\usepackage{graphicx}
\usepackage{amsmath}
\usepackage{amssymb}
\usepackage{amsfonts}
\usepackage{amsthm}
\usepackage{amstext}
\usepackage{amsbsy}
\usepackage{amsopn}
\usepackage{amscd}
\usepackage{enumerate}
\usepackage{color}
\usepackage[colorlinks]{hyperref}
\usepackage[hyperpageref]{backref}

\newtheorem{theorem}{Theorem}

\newtheorem{lemma}[theorem]{Lemma}

\theoremstyle{definition}

\newtheorem{question}[theorem]{Question}

\newcommand{\F}{\mathbb{F}}

\newcommand{\Z}{\textbf{Z}}

\newcommand{\m}{\textsf{m}}
\newcommand{\s}{\textsf{s}}

\newcommand{\Sz}{\mathsf{Sz}}

\newcommand{\GL}{\mathsf{GL}}

\newcommand{\PG}{\mathsf{PG}}

\newcommand{\GF}{\textrm{GF}}

\newcommand{\Pmc}{\mathcal{P}}
\newcommand{\Bmc}{\mathcal{B}}
\newcommand{\Dmc}{\mathcal{D}}

\renewcommand{\varphi}{\phi}

\renewcommand{\leq}{\leqslant}
\renewcommand{\geq}{\geqslant}

\newcommand{\imod}[1]{\allowbreak\mkern4mu({\operator@font mod}\,\,#1)}

\begin{document}

\title[Suzuki-Tits ovoid designs]{A note on two families of $2$-designs arose from Suzuki-Tits ovoid}

\author[S.H. Alavi]{Seyed Hassan Alavi}
\address{S.H. Alavi, Department of Mathematics, Faculty of Science, Bu-Ali Sina University, Hamedan, Iran}
\email{alavi.s.hassan@gmail.com (preferred)}
\email{alavi.s.hassan@basu.ac.ir}


\maketitle%

\begin{abstract}
    In this note, we give a precise construction of one of the families of $2$-designs arose from studying  flag-transitive $2$-designs with parameters $(v,k,\lambda)$ whose replication numbers $r$ are coprime to $\lambda$. We show that for a given positive integer $q=2^{2n+1}\geq 8$, there exists a $2$-design with parameters $(q^{2}+1,q,q-1)$ and the replication number $q^{2}$ admitting the Suzuki group $\Sz(q)$ as its automorphism group. We also construct a family of $2$-designs with parameters $(q^{2}+1,q(q-1),(q-1)(q^{2}-q-1))$ and the replication number $q^{2}(q-1)$ admitting the Suzuki groups $\Sz(q)$ as their automorphism groups. 
\end{abstract}

\section{Introduction}

An \emph{ovoid} in $\PG_{3}(q)$ with $q>2$, is a set of $q^{2} + 1$ points such that no three of which are collinear. The classical example of an ovoid in $\PG_{3}(q)$ is an elliptic quadric. If $q$ is odd, then all ovoids are elliptic quadrics, see \cite{a:Barlotti-55,a:Panella-55}, while in even characteristic, there is only one known family of ovoids that are not elliptic quadrics in which $q\geq 8$ is an odd power of $2$. These were discovered by Tits \cite{a:Tits-ovoid}, and are now called the \emph{Suzuki-Tits ovoids} since the Suzuki groups naturally act on these ovoids. The main aim of this note is to introduce two infinite families of $2$-designs arose from Suzuki-Tits ovoid whose automorphism groups are the Suzuki groups. 
A $2$-design $\Dmc$ with parameters $(v,k,\lambda)$ is a pair $(\Pmc,\Bmc)$ with a set $\Pmc$ of $v$ points and a set $\Bmc$ of $b$ blocks such that each block is a $k$-subset of $\Pmc$ and each two distinct points are contained in $\lambda$ blocks. 
The number of blocks incident with a given point is a constant number $r:=bk/v$ called the \emph{replication number} of $\Dmc$. 
If $v=b$ (or equivalently, $k=r$), then $\Dmc$ is called \emph{symmetric}. 
Further definitions and notation can be found in \cite{b:Beth-I,b:Atlas,b:Praeger2018}.

Our motivation comes from a recent  classification of flag-transitive $2$-designs whose replication numbers are coprime to $\lambda$, \cite{a:ADM-AS-CP}. Excluding thirteen sporadic examples, we have found seven possible infinite families of  $2$-designs with this property, three of which are new and the rest are well-known structures, namely, \emph{point-hyperplane designs}, \emph{Witt-Bose-Shrikhande spaces} \cite{a:BDD-1988}, \emph{Hermitian Unital spaces} \cite{a:Kantor-85-Homgeneous} and \emph{Ree Unital spaces} \cite{a:Luneburg-66}. 
These new possible $2$-designs arose from studying $2$-designs admitting finite almost simple exceptional automorphism groups of Lie type, see  \cite{a:A-Exp-CP}.  Although, we have provided examples of these new $2$-designs with smallest possible parameters \cite[Section 2]{a:A-Exp-CP}, at the time of writing \cite{a:A-Exp-CP}, we have not been aware of any generic construction of these incidence structures. In \cite{a:AD-Ree}, Daneshkhah have constructed two of these infinite families of $2$-designs admitting Ree groups as their automorphism groups, and so for the remaining possibility, one may ask the following question: 

\begin{question}\label{question}
    For a given  prime power $q=2^{2n+1}\geq 8$, does there exist a $2$-design with parameters $(q^{2}+1,q,q-1)$ with replication number $q^{2}$ admitting the Suzuki group $\Sz(q)$ as its automorphism group? 
\end{question}

In this note, we aim to give a positive answer to Question \ref{question}. Indeed, in Theorem \ref{thm:main} below, we explicitly construct such a $2$-design using the natural action of Suzuki groups on Suzuki-Tits ovoids. We would also remark here that these designs can also be constructed geometrically by taking points as ovoids in $\PG_{3}(q)$ and blocks as pointed conics minus the distinguished points. We also construct a $2$-design with parameters $(q^{2}+1,q(q-1),(q-1)(q^{2}-q-1))$ admitting the Suzuki group $\Sz(q)$ as its flag-transitive automorphism group.
We call such $2$-designs \emph{Suzuki-Tits ovoid designs} as they arose from Suzuki-Tits ovoid. We note that Suzuki-Tits ovoid designs have $q(q^{2}+1)$ number of blocks, and so they are not symmetric. In view of \cite{a:Kantor-75-2-trans,a:Kantor-85-2-trans}, the  Suzuki-Tits ovoid designs are non-symmetric $2$-designs with doubly transitive automorphism groups on points.     

\section{Preliminaries}

The Suzuki groups were discovered by Suzuki \cite{a:Suzuki-60}, and a geometric construction of these groups was given by Tits \cite{a:Tits-ovoid}. We mainly follow the description of these groups from \cite[Section XI.3]{b:Huppert1982-III} with a few exceptions in our notation, see also \cite{a:Suzuki-60,a:Suzuki-62,a:Tits-ovoid}.

Let $\F=\GF(q)$ be the finite field of size $q=2^{2n+1}\geq 8$, and let $\theta$ be the automorphism of $\F$ mapping $\alpha$ to $\alpha^{r}$, where $r =\sqrt{2q}=2^{n+1}$. Therefore, $\theta^{2}$ is the Frobenius automorphism $\alpha \mapsto \alpha^{2}$. Let  $e$ be the identity $4\times 4$ matrix, and let $e_{ij}$ be the $4\times 4$ matrix with $1$ in the entry $ij$ and $0$ elsewhere.
For $x,y \in \F$ and $\kappa\in \F^{\times}$, define 
\begin{align}
\nonumber \s(x,y)=&e+xe_{21}+ye_{31}+x^{\theta}e_{32}+\\
&(x^{1+\theta}+xy+y^{\theta}) e_{41}+(x^{1+\theta}+y)e_{42}+xe_{43};\label{s}\\
\m(\kappa)=&\kappa^{1+2^{n}}e_{11}+\kappa^{2^{n}}e_{22}+\kappa^{-2^{n}}e_{33}+\kappa^{-1-2^{n}}e_{44};\label{m}\\
\tau:=& e_{14}+e_{23}+e_{32}+e_{41}. \label{tau}
\end{align}
We know that $\s(x,y)\cdot \s(z,t) = \s(x + z, y + t+x^{\theta}z)$, and hence the set $Q$ of  the matrices $\s(x,y)$ is a group of order $q^{2}$. The set of matrices $\m(\kappa)$ forms a
cyclic group $M\cong \F^{\times}$ of order $q-1$. Since
$\m(\kappa)^{-1}\cdot \s(x, y) \cdot \m(\kappa) = (x\kappa, y\kappa^{1+\theta})$,
the group $H$ generated by $Q$ and $M$ is a semidirect product of a normal subgroup $Q$ by a complement $M$, and so it has order $q^{2}(q-1)$. The Suzuki group $\Sz(q)$ is a subgroup of $\GL_{4}(q)$ generated by $H$ and the $4\times 4$ matrix $\tau$ defined as in \eqref{tau}. In what follows, $G$ will denote the Suzuki group $\Sz(q)$ with $q=2^{2n+1}\geq 8$.

The Suzuki group $G$ naturally acts on the projective space $\PG_{3}(q)$ via $[w]^{x}:=[w^{x}]$, for all $x\in G$ and $[w]\in \PG_{3}(q)$. In fact, $G$ acts as a doubly transitive permutation group of degree $q^{2} + 1$ on the Suzuki-Tits ovoid
\begin{align}\label{points}
\Pmc= \{p(\alpha,\beta) \mid  \alpha, \beta \in \F\}\cup \{\infty\} \subseteq \PG_{3}(q),
\end{align}
where $p(\alpha,\beta)=[\alpha^{2+\theta} + \alpha\beta + \beta^{\theta} , \beta, \alpha, 1]$ and $\infty := [1,0,0,0] \in \PG_{3}(q)$. In particular, the action of the matrices $\s(x,y)$ and $\m(\kappa)$ on the projective points $p(\alpha,\beta)$ of $\Pmc\setminus \{\infty\}$ can be explicitly written as follows
\begin{align}
p(\alpha,\beta)^{\s(x,y)}&= p(\alpha+x,\beta+y+\alpha x^{\theta}+x^{1+\theta}),\label{ps}\\
p(\alpha,\beta)^{\m(\kappa)}&=p(\alpha \kappa,\beta\kappa^{1+\theta})\label{pm}. 
\end{align}

Note that $H$ fixes $\infty$, and hence $H$ is the point-stabilizer  $G_{\infty}$.  The subgroup $H$ of $G$ acts as a Frobenius group on $\Pmc\setminus\{\infty\}$, where $Q$ is the Frobenius kernel of $H$ acting regularly on $\Pmc\setminus\{\infty\}$, and $M=G_{\infty,\omega}$ is the Frobenius complement of $H$ fixing the second point $\omega := p(0,0)=[0,0,0,1] \in \Pmc$
and acting semiregularly on $\Pmc \setminus  \{\infty, \omega\}$. Therefore, the stabilizer of any three points in $\Pmc$ is the  trivial subgroup. Moreover, the map $Hg\mapsto \infty^{g}$ induces a permutational isomorphism between the $G$-action on the set of right cosets of $H$ in $G$ and the $G$-action on the Suzuki-Tits avoid $\Pmc$. 

We now consider the subgroup $Q_{0}$ of $Q$ consisting of all matrices $\s(0,y)$. Note that the matrices of the form $\s(0,y)$ are the only involutions in $Q$. Thus $Q_0$ is a normal subgroup of $Q$ of order $q$. Moreover, $Q_{0}=Q'=\Z(Q)$. Let $K$ be the subgroup of $H$ generated by $Q_{0}$ and $M$, that is to say,
\begin{align}\label{K}
K=\langle \s(0,y),\m(\kappa)\mid y\in \F \text{ and }\kappa\in \F^{\times}\rangle,
\end{align}
where $\s(0,y)$ and $\m(\kappa)$ are as in \eqref{s} and \eqref{m}, respectively.
Then $K$ is a Frobenius group of order $q(q-1)$ whose Frobenius kernel and Frobenius complement are $Q_{0}$ and $M$, respectively. Let now 
\begin{align}\label{orbits}
\Delta_{1}=\{\infty\}, \  \Delta_{2}= \{p(0,\beta) \mid  \beta \in \F\}, \text{ and } \Delta_{3}=\{p(\alpha,\beta) \mid \alpha\in\F^{\times},\, \beta \in \F\}.
\end{align} 
Then $|\Delta_{1}|=1$, $|\Delta_{2}|=q$ and $|\Delta_{3}|=q(q-1)$. Assuming these notation, we prove the following key lemma.

\begin{lemma}\label{lem:orbs}
    The subsets $\Delta_{1}$, $\Delta_{2}$ and $\Delta_{3}$ defined in \eqref{orbits} are the only orbits of $K$ in its action on the ovoid $\Pmc$. Moreover, $K$ is the setwise-stabilizer $G_{\Delta_{i}}$, for $i=2,3$.  
\end{lemma}  
\begin{proof}
    Since $K$ is a subgroup of $H$ fixing $\infty$, it follows that $\Delta_{1}$ is an orbit of $K$. It is easily followed by \eqref{pm} that the $\Delta_{i}$, for $i\in\{2,3\}$, are $M$-invariant subsets of $\Pmc$. Moreover, by \eqref{ps}, we have that $p(\alpha,\beta)^{\s(0,y)}=p(\alpha,\beta+y)$, and so the $\Delta_{i}$ are also $Q_{0}$-invariant. Therefore, by \eqref{K}, we conclude that the $\Delta_{i}$ are $K$-invariant subsets of $\Pmc$. Further, if $\alpha\in \F^{\times}$, then it follows from  \eqref{ps} and \eqref{pm} that  $p(0,0)^{\s(0,\beta)}=p(0,\beta)$ and $p(1,0)^{\m(\alpha)\s(0,\beta)}=p(\alpha,\beta)$. This implies that $K$ is transitive on each $\Delta_{i}$, and since the set of $\Delta_{i}$, for $i=1,2,3$, forms a $K$-invariant partition of $\Pmc$, we conclude that the $\Delta_i$ are all distinct $K$-orbits on $\Pmc$. 
    
    We now prove that $G_{\Delta_{2}}=K$. Obviously, $K$ is a subgroup of $G_{\Delta_{2}}$. Recall that $G$ is generated by $\s(x,y)$, $\m(\kappa)$ and $\tau$ defined as in \eqref{s}, \ref{m} and \eqref{tau}, respectively. The fact that $p(0,0)^{\tau}=\infty$ implies that $\tau$ does not fix $\Delta_{2}$. Hence, $G_{\Delta_{2}}$ is a subgroup of $H=\langle \s(x,y),\m(\kappa) \mid  x,y\in \F \text{ and }\kappa\in \F^{\times} \rangle$. We have also proved that $K=Q_{0}M$ fix $\Delta_{2}$, and so $K\leq G_{\Delta_{2}}\leq H$. If a generator $g:=\s(x,y)$ fixes $\Delta_{2}$, then \eqref{ps} yields $p(x,\beta+y+x^{1+\theta})=p(0,\beta)^{\s(x,y)}\in \Delta_{2}$, and so  $x=0$, that is to say, $g=\s(0,y)\in Q_{0}$. Recall that $M$ fixes $\Delta_{2}$, and hence we conclude that  $G_{\Delta_{2}}=K$. We finally show that $G_{\Delta_{3}}=K$. By the same argument as in the previous case, $G_{\Delta_{3}}$ is a subgroup of $H$ containing $K$. Let $\s(x,y)\in H$ fix $\Delta_{3}$. Then $p(\alpha,\beta)^{\s(x,y)}\in \Delta_{3}$ for all $\alpha\neq 0$, and so by \eqref{ps}, this is equivalent to $x\neq \alpha$, for all $\alpha\neq 0$. Thus, $x=0$, and hence  $G_{\Delta_{3}}$ is a subgroup of  $K=Q_{0}M$ implying that $G_{\Delta_{3}}=K$.             
\end{proof}

\section{Existence of Suzuki-Tits ovoid designs}

In this section, we prove our main result Theorem~\ref{thm:main} and construct two infinite families of $2$-designs admitting $G = \Sz(q$) as their automorphism groups. In order to construct out favourite designs, we use \cite[Proposition III.4.6]{b:Beth-I}. In fact, if $G$ is a doubly transitive permutation group on a finite set $\Pmc$ of size $v$ and $B$ is a subset of $\Pmc$ of size $k\geq 2$, then the incidence structure $\Dmc=(\Pmc,B^{G})$ is a $2$-design with parameters $(v,k,\lambda)$ with automorphism group $G$, where $B^{G}=\{B^{x}\mid x\in G\}$. The design $\Dmc$ has $b=|G:G_{B}|$ number of blocks and $\lambda$ is equal to $bk(k-1)/v(v-1)$. We now prove our main result.

\begin{theorem}\label{thm:main}
    Let $G=\Sz(q)$ with $q=2^{2n+1}\geq 8$. Let also $\Pmc$ be the Suzuki-Tits ovoid defined as in \eqref{points}, and let $\Delta_{i}$ be as in \eqref{orbits}, for $i\in\{2,3\}$. Then
    \begin{enumerate}[\rm (a)]
        \item if $\Bmc_{2}=\Delta_{2}^{G}$, then $(\Pmc,\Bmc_{2}) $ is a $2$-design with parameters $(q^{2}+1,q,q-1)$ and the replication number $q^{2}$;
        \item if $\Bmc_{3}=\Delta_{3}^{G}$, then $(\Pmc,\Bmc_{3}) $ is a $2$-design with parameters $(q^{2}+1,q(q-1),(q-1)(q^{2}-q-1))$ and the replication number $q^{2}(q-1)$;
    \end{enumerate}
    Moreover, the Suzuki group  $\Sz(q)$ is a flag-transitive automorphism group of the designs in parts {\rm (a)} and {\rm (b)} acting primitively on the  points set $\Pmc$ but imprimitively on the blocks set $\Bmc_{i}$. 
\end{theorem}
\begin{proof}
    By the fact that $G$ is doubly transitive on $\Pmc$, \cite[Proposition III.4.6]{b:Beth-I} implies that the incidence structures $\Dmc_{i}=(\Pmc,\Bmc_{i})$ are  $2$-designs with parameters $(v,k_{i},\lambda_{i})$, for $i\in\{2,3\}$, admitting $G=\Sz(q)$ as their automorphism groups. Recall that $v=|\Pmc|=q^{2}+1$. Moreover, by Lemma \ref{lem:orbs}, we have that $|G_{\Delta_{i}}|=|K|=q(q-1)$, for $i\in\{2,3\}$, where $K$ is subgroup of $G$ defined as in \eqref{K}. Therefore,  the designs $\Dmc_{i}$ in parts (a) and (b) have (the same) number of blocks $b=|G:G_{\Delta_{i}}|=q(q^{2}+1)$. Note that $k_{i}=|\Delta_{i}|$ and   $\lambda_{i}=bk_{i}(k_{i}-1)/v(v-1)$, and the replication number $r_{i}$ is equal to $bk_{i}/v$, for $i\in\{2,3\}$. Therefore, $\Dmc_{2}$ is a $2$-design with parameters $(q^{2}+1,q,q-1)$ as in part (a) and $\Dmc_{3}$ is a $2$-design with parameters $(q^{2}+1,q(q-1),(q-1)(q^{2}-q-1))$ as in part (b). Since $G$ is transitive on $\Bmc_{i}$ and $\Delta_{i}$ is an orbit of $K=G_{\Delta_{i}}$, we conclude that $G$ is a flag-transitive automorphism group of $\Dmc_{i}$, for $i\in\{2,3\}$. The group $G$ is primitive on $\Pmc$ as it is  doubly transitive  on $\Pmc$, however, $G$ is imprimitive on $\Bmc_{i}$ as the block-stabilizer $K=G_{\Delta_{i}}$ is not a maximal subgroup of $G$.     
\end{proof}

We remark here that the design constructed in Theorem \ref{thm:main}(a) introduces an infinite family of examples of $2$-designs with $\gcd(r,\lambda)=1$ obtained in \cite[Theorem 1.1(a)]{a:A-Exp-CP} and gives a positive answer to Question \ref{question}. 

\section*{Acknowledgment}
The author would like to thank John Bamberg for pointing out the geometric construction of the designs with parameters in Question \ref{question}.  





\end{document}